\documentclass[a4 paper, 11pt]{amsart}
\usepackage[utf8]{inputenc}
\usepackage{amssymb, hyperref, enumerate, color}

\newtheorem{theorem}{Theorem}[section]
\newtheorem{proposition}[theorem]{Proposition}
\newtheorem{lemma}[theorem]{Lemma}
\newtheorem{corollary}[theorem]{Corollary}

\newtheorem{example}[theorem]{Example}

\theoremstyle{remark}
\newtheorem{remark}[theorem]{Remark}

\numberwithin{equation}{section}

\DeclareMathOperator{\esssup}{ess-sup}

\DeclareMathOperator{\real}{Re}
\DeclareMathOperator{\imag}{Im}

\newcommand{\R}{\mathbb{R}}
\newcommand{\N}{\mathbb{N}}
\newcommand{\C}{\mathbb{C}}
\newcommand{\Q}{\mathbb{Q}}
\newcommand{\F}{\mathcal{F}}
\newcommand{\Hi}{\mathcal{H}}
\newcommand{\B}{\mathcal{B}}

\newcommand{\dih}{\int_\Omega^\oplus H_s \,d\mu(s)}

\title[Direct Integrals of Operator Semigroups]{Direct Integrals of Strongly Continuous Operator Semigroups}

\author[A.C.S.\ Ng]{Abraham C.S.\ Ng}
\address[A.C.S.\ Ng]{St Edmund Hall, Queen's Lane, Oxford OX1 4AR, UK}
\email{abraham.ng@maths.ox.ac.uk}


\begin{document}
	
	\begin{abstract}
		The goal of this article is to develop a theory for direct integrals of $C_0$-semigroups on Hilbert spaces parallel to the recent approach by Lachowicz and Moszy\'{n}ski for direct sums of Banach spaces, diagonal operators, and semigroups. In it we deal with the existence and characterisation of semigroups, asymptotic rates, and questions of decomposability. 
	\end{abstract}
	
\maketitle

\section{Introduction}\label{sec:1}

The study of strongly continuous semigroups, or $C_0$-semigroups, was motivated by so-called Cauchy problems -- initial value problems regarding evolution equations. This theory applied abstract functional and operator theoretic machinery to the solving of partial differential equations (PDEs) and related problems (see \cite{ABHN,EN}). There has been a renewed interest in $C_0$-semigroups due to the discovery of abstract methods that can be used to address asymptotic questions such as the energy decay of physical systems. These have come in the form of the so-called quantified Tauberian theorems, beginning with \cite{BaDu}, exploding upon the PDE scene through  \cite{BoTo}, and largely perfected in \cite{RSS}.

In this article, we combine a recent approach to direct sums of $C_0$-semigroups, done by Lachowicz and Moszy\'nski \cite{LM}, with direct integral theory. Indeed, one can view our results as a generalisation and expansion of those in \cite{LM} when considering the setting of Hilbert spaces.

Direct integral theory has been around for a long time, originating from the study of operator algebras \cite{Dix,vN} in the 1940s-50s. In particular, Nussbaum's generalisation of von Neumann's reduction theory for bounded operators \cite{vN2} to unbounded operators \cite{Nuss} motivated much further work within direct integral theory in areas such as spectral theory \cite{Azoff,Chow,Lennon} and functional calculus \cite{CG,Gil} in the 1960s-70s. A book was compiled by Nielsen in 1980 \cite{Niels}, but since then, pure operator theory on direct integrals does not seem to have been developed as frequently. More recent additions to the theory are those of \cite{DNSZ} on $W^*$-algebras and \cite{GGST} on weak measurability.

It is somewhat surprising then, that although the development of $C_0$-semigroup theory substantially overlapped in time with that of direct integral theory, there appears to be a gap in the literature when it comes to comprehensively combining these two theories. This peculiarity is further compounded by the fact that both theories have already been largely completed for some time. Indeed, it is worth remarking that no comprehensive treatment on direct sums of $C_0$-semigroups seems to have been produced until 2016 in \cite{LM} (even though the use of infinite direct sums was already known as a tool to semigroup theorists). The closest that previous work has come to filling this gap in the literature is the study of semigroups of partial isometries that decompose into direct integrals of truncated shifts (see \cite{EW,ELW,Martin,PR,W}). This article aims to rectify this situation and properly fill the mostly empty space. Furthermore, we hope to provide an interesting first step towards further developments and applications of direct integral and $C_0$-semigroup theory in the direction of both Cauchy problems and operator algebras.

The central results of this article are the following two theorems in Section \ref{sec:4}. Theorem \ref{sgthm} states that the direct integral of a family of $C_0$-semigroups is itself a $C_0$-semigroup on the direct integral space if and only if the family is uniformly exponentially bounded. Theorem \ref{sgthmconv} states that if a decomposable operator generates a $C_0$-semigroup, then the individual operators comprising this original decomposable operator also individuallly generate $C_0$-semigroups, the direct integral of which coincides with the original $C_0$-semigroup. In particular, any $C_0$-semigroup on the direct integral space that is generated by a decomposable operator is itself decomposable. This is particularly interesting, given that there is a distinction between so-called maximally defined operators and those defined via a direct integral (see \cite{GGST} where this distinction is teased out).

Sections \ref{sec:2}-\ref{sec:3} provide the preliminaries and basic operator-theoretic results needed to tackle the central issue, which, as mentioned, is addressed in Section \ref{sec:4}. Sections \ref{sec:5}-\ref{sec:7} then deal with special cases, examples, and asymptotic questions that arise out of our central results, including a discussion on a theorem by Maniar and Nafiri \cite{MaNa} and how it follows as a simple corollary of our more general quantified asymptotics result.

From a $C_0$-semigroup theoretic perspective, our results are already deeply interesting, answering natural fundamental questions and laying the groundwork for applications to Cauchy problems, quantified asymptotic theory, and rates of decay. However, this article may also lead to two other rich areas of application, yet to be explored.

Firstly, since every von Neumann algebra on a separable Hilbert space is a direct integral of factors \cite{vN}, we hope that further progress can be made in the theory of quantum Markov semigroups (QMSs) on von Neumann algebras by building on our results for $C_0$-semigroups on direct integral spaces. QMSs, a natural generalisation of classical Markov semigroups that have interested mathematical physicists for almost fifty years, are a certain type of one-parameter operator semigroup on von Neumann algebras, originally arising out of the study of open quantum systems. Recently, there has been an interest in deriving results concerning QMSs and their generators inspired from classical strongly continuous semigroup theory (see \cite{AZ1,AZ2} for example). The key difference between QMSs and $C_0$-semigroups is that QMSs act on what are essentially spaces of operators (subspaces of $\B(H)$, the space of bounded linear operators on a Hilbert space $H$) whereas $C_0$-semigroups act directly on a Banach space $X$, leading to the use of different topologies. Nonetheless, ideas from $C_0$-semigroup theory evidently can be useful in developing an understanding of QMSs. On the theory of QMSs and their geneators, see \cite{Davies,Kraus,Lind,Stine} to name but a few milestone papers.

Secondly, direct integrals, often called fibre integrals, have recently re-appeared independently from the theory of von Neumann algebras, in the area of homogenisation theory. For examples of this, see \cite{BiSu,ChWa,CoWa,NiPa}. This provides a potential avenue for more applications of our results. See also \cite[Section 5]{Ng} for an area of potential overlap between homogenisation theory (and hence direct integral theory) and quantified asymptotics for $C_0$-semigroups.

\subsection*{Acknowledgements} The author thanks Charles Batty, David Seifert, and Stuart White for helpful discussions on the topic and production of this article. He is also grateful to the University of Sydney for funding this work through the Barker Graduate Scholarship. A further especially warm thank you goes to the reviewer for their careful reading and constructive feedback because of which this article is much improved.

\section{Preliminaries}\label{sec:2}

In this article, we use standard notation, denoting the domain, spectrum, and resolvent set of a closed linear operator $B$ acting on a Hilbert space $X$ (assumed always to be complex) by $D(B)$, $\sigma(B)$, and $\rho(B)$ respectively. The resolvent operator $(\lambda-B)^{-1}$, for $\lambda \in \rho(B)$, will usually be denoted by $R(\lambda,B)$. We say that $B$ is densely defined if $D(B)$ is dense in $X$. We also write $\B(X)$ for the space of bounded linear operators on $X$.

The following proposition is a simple fact about resolvents that will be used in the article, stated without proof and an immediate corollary of the Neumann series expansion.

\begin{proposition}\label{denseresolvent}
	Let $D$ be a dense subset of $E\subset \C$. Suppose that there exists $M>0$ such that for all $\lambda \in D$, $\lambda \in \rho(B)$ and $\|R(\lambda,B)\|\leq M$. Then $E\subset \rho(B)$ and $\|R(\lambda,B)\|\leq M$ for all $\lambda \in E$.
\end{proposition}

We will also use $1_G$ to denote the characteristic function for a set $G$ and the abbreviation `a.e.'\ to mean either `almost every' or `almost everywhere', depending on the context.

Turning now to direct integrals of Hilbert spaces, let $\Omega$ be a locally compact Hausdorff topological space and $\mu$ a $\sigma$-finite positive Borel measure on $\Omega$. An important example of this is the counting measure on $\N$ with the discrete topology.

A \textit{Hilbert (space) bundle} over \textit{base space} $\Omega$ is a pair $(\mathcal{H},\pi)$ where $\pi:\mathcal{H}\to\Omega$ is a surjection such that for all $s\in\Omega$, the \textit{fiber} of $\mathcal{H}$ at $s$, defined by $H_s= \pi^{-1}(s)$, is a Hilbert space with an inner product $\langle\cdot,\cdot\rangle_s$ and induced norm $\|\cdot\|_s$. A \textit{section} $x$ of $\Hi$ is a function $x:\Omega \to \mathcal{H}$ such that $\pi\circ x(s) = s$ for all $s\in\Omega.$ Henceforth, $\pi$ will be notationally supressed and the Hilbert space bundle will be referred to by $\Hi$ alone. A \textit{measurable field of Hilbert spaces} is a pair $(\mathcal{H},\mathcal{F})$ where $\mathcal{H}$ is a Hilbert bundle and $\mathcal{F}$ is a collection of sections satisfying the following conditions:

\begin{enumerate}
	\item For all $x,y \in \mathcal{F}$, the $\C$-valued function $s \to \langle x(s),y(s)\rangle_s$ is $\mu$-measurable;
	\item If $z$ is a section such that for all $x\in \F$, the function $s \to \langle x(s),z(s)\rangle_s$ is $\mu$-measurable, then $z \in \F$; and
	\item There exists a sequence $\{\xi_i\}_{i=1}^\infty \subset \F$ called a \textit{fundamental sequence} for $\Hi$ such that for all $s \in \Omega$, the elements $\{\xi_i(s)\}_{i=1}^\infty$ are dense in $H_s$.
\end{enumerate}
$\F$ is called a choice of \textit{measurable sections} of $\Hi$.

Given a measurable field of Hilbert spaces $(\Hi,\F)$ with associated measure $\mu$, the \textit{direct integral}
$$\int_{\Omega}^\oplus H_s\, d\mu(s)$$ is the Hilbert space of all measurable sections $x \in \F$ such that
$$\|x\|^2 = \int_\Omega \|x(s)\|_s^2 \, d\mu(s) < \infty$$ modulo measurable sections that are $0$ $\mu$-a.e.\ and equipped with the inner product
$$\langle x,y\rangle = \int_\Omega \langle x(s),y(s)\rangle_s \,d\mu(s).$$

We state some simple properties without proof (see \cite[Section 3]{Ga} and \cite[Lemma 35]{Y}).

\begin{proposition}\label{sim}
	Let $\mu$  and $\Omega$ be as above. The following statements are true.
	\begin{enumerate}[(i)]
		\item If $\mathcal{M}=\{M_k\}_{k=1}^\infty$ is a countable collection of sets with finite measure such that $\Omega = \bigcup_{k=1}^\infty M_k$, the triple sequence given by
		$$\xi_{i,j,k}(\cdot)=\xi_i(\cdot) 1_{\{s \in \Omega : \|\xi_i(s)\|_s \leq j\}}(\cdot)1_{M_k}(\cdot)$$ is also a fundamental sequence after rearrangement into a sequence of one parameter. Furthermore, this fundamental sequence has elements that are pointwise bounded and have finite $\dih$ norm. Hence, from now on we can assume that the fundamental sequence is pointwise bounded and contained in $\dih$.
		\item There exists a sequence $\{e_i\}_{i=1}^\infty$ of measurable sections such that for every $s \in \Omega$, the vectors $e_i(s)$ form an orthonormal basis for the fiber $H_s$.
		\item $\{1_{M_k}e_i : i,k \in \N\}$ is a countable subset such that its span is dense in $\int_\Omega^\oplus H_s\, d\mu(s)$. Hence, $\dih$ is separable.
	\end{enumerate}
\end{proposition}

Given $H = \dih,$ if $\Omega' \subset\Omega$ is measurable and $\mu(\Omega') >0$, then we can form the direct integral $\int_{\Omega'}^{\oplus}H_s\,d\mu(s)$ using the same definition as before, with $\Omega'$ replacing $\Omega$. $\int_{\Omega'}^{\oplus}H_s\,d\mu(s)$ can be identified with the subspace of $H$ given by
$$\left\{x \in \dih : x(s) = 0 \text{ for a.e.\ } s \notin \Omega'\right\}.$$
$\int_{\Omega'}^\oplus H_s \,d\mu(s)$ is closed, since if $\{x_n\}_{n=1}^\infty \subset \int_{\Omega'}^\oplus H_s \,d\mu(s)$ and $x_n \to x$ in $H$, then there exists a subsequence such that $x_{n_k}(s) \to x(s)$ in $H_s$ for a.e.\ $s \in \Omega$. In particular, $x(s) = 0$ for a.e.\ $s \notin \Omega'$. Hence $x \in \int_{\Omega'}^\oplus H_s\, d\mu(s)$.

\section{Direct Integrals of Operators}\label{sec:3}

In most discussions on direct integrals, for a given measurable field $(\Hi,\F)$, a family of bounded operators $\{T(s) \in \B(H_s) : s \in \Omega\}$ forms a \textit{measurable field of operators} if the function
$$s \mapsto \langle T(s)x(s),y(s)\rangle_s$$ is measurable for all $x,y \in \F$. Some treatments also require that $$s\mapsto \|T(s)\|_{\B(H_s)}$$ be essentially bounded. When this extra condition is satisfied, we shall call $\{T(s) \in \B(H_s) : s \in \Omega\}$ a \textit{bounded measurable field of operators}. We will not go into a discussion of measurable fields and direct integrals of bounded operators, these can be found in \cite{Dix,Niels,Take}.

The standard generalisation to closed operators, first given by Nussbaum \cite{Nuss}, is through the characteristic matrices introduced by Stone \cite{St}. The characteristic matrix $(P_{i,j})$ of a closed operator $A$ on a Hilbert space $H$ is the $2 \times 2$ matrix of bounded operators representing the projection $P$ of $H \times H$ onto the closed subspace $\Gamma(A)$, the graph of $A$. Nussbaum defined a \textit{measurable field of closed operators} to be a family of closed operators $\{A(s) :D(A(s))\subset H_s \to H_s : s \in \Omega\}$ such that for each $i,j$, the family $\{P_{i,j}(s) : s \in \Omega\}$ forms a measurable field of bounded operators. This is shown to be consistent with the bounded case in \cite[Proposition 6]{Nuss}

We provide an alternative definition which turns out to be equivalent modulo a resolvent condition. We define a \textit{measurable field of closed operators} to be a family of closed operators $\{A(s) :D(A(s))\subset H_s \to H_s : s \in \Omega\}$ satisfying the requirement that there exists a fixed $\nu \in \C$ such that for a.e.\ $s\in\Omega$, $\nu \in \rho(A(s))$ and the family of resolvents $\{R(\nu,A(s)) : s\in \Omega\}$ adjusted on a set of measure zero is a bounded measurable field of operators.

When there exists such a $\nu \in \C$, the equivalence of definitions follows from the simple fact that adding scalar multiples of the identity does not change measurability in either definition and that inverses of Nussbaum measurable fields are measurable when they exist because the characteristic matrix of the inverse is merely a rearrangement of the components (see \cite[Proposition 5]{Nuss}). Note that since
$$R(\lambda,B) = (\lambda - \nu + R(\nu,B)^{-1})^{-1},\quad (\lambda,\nu \in \rho(B)),$$ for any operator $B$, $\{R(\nu,A(s)) : s\in \Omega\}$ is Nussbaum measurable if and only if $\{R(\lambda,A(s)) : s\in \Omega\}$ is Nussbaum measurable for all $\lambda,\nu \in \rho(A(s))$ for a.e.\ $s\in\Omega$.

We can now define the \textit{direct integral of a measurable field of closed operators} $\{A(s):D(A(s)) \subset H_s \to H_s : s \in \Omega\}$ on the maximal reasonable domain in $\int_\Omega^\oplus H_s \,d\mu(s)$ as the operator $A$ given by
\begin{align*}
	D(A) & = \bigg\{x \in \int_\Omega^\oplus H_s \,d\mu(s) : x(s) \in D(A(s)) \text{ for a.e.\ } s \in \Omega, \\ & \hspace{15 em} A(\cdot)x(\cdot) \in \int_\Omega^\oplus H_s\, d\mu(s)\bigg\}, \\
	A & = A(\cdot)x(\cdot), \quad x \in D(A),
\end{align*} and denoted by $\int_\Omega^\oplus A(s)\, d\mu(s)$.

In accordance with this definition, if we write $A = \int_{\Omega}^{\oplus}A(s)\,d\mu(s)$, there is an implicit assumption that there exists at least one $\lambda\in\C$ such that for a.e.\ $s\in\Omega$ $\lambda \in \rho(A(s))$ and that $s\mapsto\|R(\lambda,A(s))\|_{\B(H_s)}$ is essentially bounded.

Given $A = \int_{\Omega}^{\oplus}A(s)\,d\mu(s)$, if $\Omega' \subset\Omega$ is measurable and $\mu(\Omega') > 0$, then we can form the direct integral of closed operators $A_{\Omega'} = \int_{\Omega'}^{\oplus}A(s)\,d\mu(s)$ on $\int_{\Omega'}^{\oplus}H_s\,d\mu(s)$ using the same definitions as before, with $\Omega'$ replacing $\Omega$. When considering $\int_{\Omega'}^{\oplus}H_s\,d\mu(s)$ as a closed subspace of $\dih$, $A_{\Omega'}$ coincides with the restriction of $A$ to $\int_{\Omega'}^{\oplus}H_s\,d\mu(s)$, that is,
$$D(A_{\Omega'}) = \int_{\Omega'}^\oplus H_s \,d\mu(s) \cap D(A) \ \ \text{ and } \ \
A_{\Omega'} = A|_{D(A_{\Omega'})}.$$

Note that there may exist operators $A$ on a direct integral space $H=\dih$ that are not themselves, the direct integral of a measurable field of closed operators. To see this, take $\Omega = \{1,2\}$ with the discrete measure and $(\Hi,\pi)$ to be the bundle given by $H_1=H_2 = \C$. Then $\C^2$ can be considered as the bundle $\Hi$. The linear operators on $\Hi$ are the $2\times 2$ matrices and only the diagonal matrices are direct integrals of operators on $H_1$ and $H_2$. In the case, however, where $A = \int_{\Omega}^{\oplus}A(s)\,d\mu(s)$ for a measurable field of closed operators $\{A(s) :D(A(s))\subset H_s \to H_s : s \in \Omega\}$, $A$ is said to be \textit{decomposable}. Decomposable operators are characterised in \cite[Corollary 4]{Nuss} as those closed operators that commute with every so-called bounded diagonalisable operator.

We now provide some basic properties of direct integrals of operators. First, we deal with boundedness.

\begin{proposition}\label{bddprop}
	Let $A = \int_{\Omega}^{\oplus}A(s)\,d\mu(s)$ be densely defined and $M\geq0$. The following are equivalent:
	\begin{enumerate}[(i)]
		\item $A(s) \in \B(H_s)$ for a.e.\ $s\in\Omega$ and $\underset{s\in \Omega}{
			\esssup}\|A(s)\|_{\B(H_s)} \leq M$.
		\item $\int_\Omega^\oplus A(s)\, d\mu(s) \in \B\left(\int_\Omega^\oplus H_s\, d\mu(s)\right)$ and $\|\int_\Omega^\oplus A(s) \,d\mu(s)\| \leq M$.
	\end{enumerate}
	Note that $s \mapsto \|A(s)\|_{\B(H_s)}$ is measurable since $\dih$ is separable so that $\|A(s)\|_{\B(H_s)}$ is a supremum of countably many measurable functions of the form $\langle A(s)x(s),A(s)x(s)\rangle_s^{1/2}$.
\end{proposition}

\begin{proof}
	Assume (i). Then for any $x \in D(A)$,
	$$\|Ax\|^2 = \int_\Omega \|A(s)x(s)\|_s^2\,d\mu(s) \leq M^2 \int_\Omega \|x(s)\|_s^2\, d\mu(s) = (M\|x\|)^2,$$ and we are done by density of $D(A)$.
	
	For the converse, we must take care that any offending function we construct is measurable. Let $Q=\{a+bi : a,b\in\Q\}$. Taking the set $\{e_i\}_{i=1}^\infty$ from Proposition \ref{sim}, $$\left\{\sum_{i=1}^n  a_i e_i : a_i \in Q, n \in \N\right\}$$ is a countable set which we relabel as $\{f_i\}_{i=1}^\infty$ such that $\{f_i(s)\}_{i=1}^\infty$ is dense in $H_s$ for all $s \in \Omega$. Let $$K_i := \left\{s\in\Omega : \|A(s)f_i(s)\|> M\|f_i(s)\|_s\right\}, \quad (i \in\N).$$ These sets are clearly measurable. If for a.e.\ $s\in \Omega$, $\|A(s)f_i(s)\|\leq M\|f_i(s)\|_s$ for all $i \in \N$, then $\|A(s)\|\leq M$ for a.e.\ $s$. Thus, if (i) fails, there exists $j\in\N$ such that $\mu(K_j) >0$.
	
	Let $x'(\cdot) := f_j(\cdot)1_{K_j}(\cdot)$ and note that it is measurable and in $\dih$. Since
	$$\|Ax'\|^2 = \int_{K_j} \|A(s)f_j(s)\|_s^2 \,d\mu(s) > M^2\int_{K_j}\|f_j(s)\|_s^2\,d\mu(s) = (M\|x'\|)^2,$$  (ii) also fails.
\end{proof}

Second, we state a theorem on adjoints and inverses.

\begin{theorem}
	Suppose that $A = \int_{\Omega}^{\oplus}A(s)\,d\mu(s)$. Then the following are true.
	\begin{enumerate}[(i)]
		\item $A^*$ exists if and only if $A(s)^*$ exists a.e.,\ in which case $$A^* = \int_{\Omega}^\oplus A(s)^*\,d\mu(s).$$
		\item $A^{-1}$ exists (as a bounded operator) if and only if $A(s)^{-1}$ exists and is essentially bounded, in which case $$A^{-1} = \int_{\Omega}^\oplus A(s)^{-1}\,d\mu(s).$$
	\end{enumerate}
\end{theorem}

\begin{proof}
	\cite[Proposition 3.5]{Chow} or \cite[Theorem 3]{Nuss} proves this theorem for the Nussbaum case. Since we have the assumption of boundedness in (ii), all that remains to show is that the resolvent conditions also hold for (i). This follows easily from the fact that $\rho(B^*) = \{\overline{\lambda} : \lambda \in \rho(B)\}$ and $R(\overline{\lambda},B^*) = R(\lambda,B)^*$ for any unbounded densely defined operator $B$.
\end{proof}

Note that this also implies that if $A = \int_{\Omega}^{\oplus}A(s)\,d\mu(s)$, $A$ is densely defined if and only if $D(A(s))$ is densely defined for a.e.\ $s\in\Omega$ since an operator $B$ is densely defined if and only if $B^{*}$ exists.

Some basic spectral properties of direct integrals of operators follow immediately from this.

\begin{corollary}\label{specprop}
	Suppose that $A = \int_\Omega^\oplus A(s) \,d\mu(s)$. The following statements are true.
	\begin{enumerate}[(i)]
		\item $\begin{aligned}[t]\rho(A) = \Big\{\lambda \in \C : \ & \lambda \in \rho(A(s)) \text{ for a.e.\ } s \in \Omega \\ & \text{ and } \underset{s\in \Omega}{\esssup}\|R(\lambda,A(s))\|_{B(H_s)} < \infty\Big\}\end{aligned}$ \\ and $$R(\lambda,A) = \int_\Omega^\oplus R(\lambda,A(s))\,d\mu(s), \quad (\lambda \in \rho(A)).$$
		\item $\begin{aligned}[t] \sigma(A) = \Big\{\lambda \in \C : & \text{ there exists } G \subset \Omega\\ &  \text{ such that } \mu(G)>0 \text{ and } \lambda \in \sigma(A(s)) \text{ for all } s \in G \Big\} \\ & \hspace{3em} \cup \Big\{\lambda \in \C : \lambda \in \rho(A(s)) \text{ for a.e.\ } s \in \Omega \\ & \hspace{6 em} \textrm{ and } \underset{s\in \Omega}{\esssup}\|R(\lambda,A(s))\|_{B(H_s)} = \infty\Big\}.\end{aligned}$
	\end{enumerate}
\end{corollary}

This corollary leads to another corollary, covering and generalising the counting measure case for Hilbert spaces in \cite[Corollary 3.7]{LM}, following its proof.

\begin{corollary}
	Suppose that $A = \int_\Omega^\oplus A(s) \,d\mu(s)$ and that there exists $\delta >0$ satisfying the following condition: If $G\subset \Omega$ with $\mu(G) >0$ and $\lambda \in \rho(A(s))$ for a.e.\ $s \in G$, then there exist $G_\lambda \subset G$ with $\mu(G_\lambda) >0$ and $\eta \in \C$ such that $\eta \in \sigma(R(\lambda,A(s)))$  and $$|\eta| \geq \delta\|R(\lambda,A(s))\|_{B(H_s)}, \quad (s\in G_\lambda).$$ Then $\sigma(A)$ is equal to the closure of \begin{align*}\Big\{\lambda \in \ \C : & \text{ there exists } K \subset \Omega \text{ such that } \mu(K)>0 \\ & \hspace{7em} \text{ and } \lambda \in \sigma(A(s)) \text{ for all } s \in K \Big\}.\end{align*}

\end{corollary}

\begin{proof}
	The inclusion `$\supseteq$' follows from Corollary \ref{specprop}(ii). For the opposite inclusion let $\lambda \in \rho(A(s))$ for a.e.\ $s$ and suppose that $$\underset{s\in \Omega}{\esssup}\|R(\lambda,A(s))\|_{B(H_s)} = \infty.$$ We are done if we show that $\lambda$ can be approximated by $z$ such that there exists $K\subset\Omega$ of positive measure with $z \in \sigma(A(s))$ for all $s \in K$.
	
	Let $\varepsilon>0$. By assumption, there exists $G \subset \Omega$ of positive measure such that $\lambda \in \rho(A(s))$ and $\|R(\lambda,A(s))\|_s >\frac{1}{\delta\varepsilon}$ for all $s \in G$. Hence, there exists $G_\lambda\subset G$ of positive measure and $\eta\in \C$ such that $\eta \in \sigma(R(\lambda, A(s))$ for all $s\in G_\lambda$ and $|\eta| > \frac{1}{\varepsilon}.$ By \cite[Proposition B.2]{ABHN}, the spectral mapping theorem for the resolvent, there exists $z\in \C$ for which $z\in \sigma(A(s))$ for all $s \in G_\lambda$ such that $(z-\lambda)^{-1}=\eta$. In particular, $|z-\lambda|< \varepsilon$ and $\mu(G_\lambda) >0.$
\end{proof}

It is worth mentioning that Azoff \cite{Azoff} and Chow \cite{Chow} delve much deeper into the spectral theory of direct integrals. In particular, \cite[Examples 4.2, 4.4]{Azoff} show the importance of essential boundedness in Corollary \ref{specprop} by demonstrating that not much can be said about the spectrum of $A$ given only the spectra of a.e.\ $A(s)$.

Finally, we turn to compactness. Before we prove the following lemma, recall that a measurable subset $A\subset \Omega$ is called an atom (with respect to $\mu$) if $\mu(A)>0$ and for all measurable $B\subset A$, either $\mu(B) = 0$ or $\mu(B) = \mu(A)$.

\begin{lemma}\label{comlem}
	If $\int_{\Omega}^\oplus A(s)\,d\mu(s)$ is compact and $\|A(s)\|_{\B(H_s)} \geq \varepsilon$ for some $\varepsilon >0$ and a.e.\ $s\in \Omega$, then $\Omega$ contains an atom.
\end{lemma}

\begin{proof}
	Assume otherwise. Then there exist subsets $\Omega = G_1 \supsetneq G_2 \supsetneq \dots$ such that $\mu(G_{i}) > \mu(G_{i+1}) > 0$ for all $i\in\N$. Writing $K_i = G_i\setminus G_{i+1}$, we get that $G = \bigcup_{i=1}^\infty K_i$ is a disjoint union and $\mu(K_i) = \mu(G_i) - \mu(G_{i+1}) > 0$ for all $i\in\N$. Using the same construction as in the proof of Proposition \ref{bddprop}, for all $i\in\N$, there exists a measurable subset $K_i' \subset K_i$ with $\mu(K_i')>0$ and a (normalised) function $y_i \in \dih$ such that $\|A(s)y_i(s)\|_s > \frac{\varepsilon}{2}$ for all $s\in K_i'$, $\|y_i(s)\|_s =1$ for all $s\in K_i'$, and $y_i(s) =0$ for all $s\notin K_i'$. Let $x_i := \mu(K_i')^{-1/2}y_i(s).$ Then $x_i \in \dih$ and $\|x_i\| = 1$ for all $i\in\N$. But
	\begin{align*}\|Ax_i - & Ax_j \|^2 \\ & = \mu(K_i')^{-1}\int_{\Omega}\|A(s)y_i(s)\|_s^2\,d\mu(s) +\mu(K_j')^{-1}\int_{\Omega}\|A(s)y_j(s)\|_s^2\,d\mu(s)\\ & > \frac{\varepsilon^2}{2}
	\end{align*} for all $i\neq j$ as $K_i' \subset K_i$ and $K_j' \subset K_j$ are disjoint. Hence, $A$ cannot be compact.
\end{proof}

Recall that for a direct sum of operators $T = \bigoplus_{i=1}^\infty T_i$ on a direct sum of Hilbert spaces $\bigoplus_{i=1}^\infty H_i$, $T$ is compact if and only if $T_i$ is compact for all $i\in\N$ and $\|T_i\|_{\B(H_i)} \to0$ (proved via finite-rank approximations). Note also that for Radon measures, if $A\subset\Omega$ is an atom, then $A$ is of the form $\{x\}\cup N$ where $x\in A$ and $\mu(N) = 0$. 

\begin{theorem}\label{comthm}
	Let $\mu$ be a Radon measure. Suppose that $A=\int_{\Omega}^\oplus A(s)\,d\mu(s)$. Then $A$ is compact if and only if $\Omega = \Omega_1 \cup \Omega_2$ where $\Omega_1$ and $\Omega_2$ are disjoint, $\Omega_1 = \{s_1,s_2,\dots\}$ is a countable set of points such that $\mu(s_i) > 0$ and $A(s_i)$ is compact and non-zero for all $i\in\N$, $\|A(s_i)\|_{\B(H_s)} \to 0$ as $i\to \infty$, and $A(s) = 0$ for a.e.\ $s\in \Omega_2$.
\end{theorem}

\begin{proof}
	The `if' direction is simply the case of a direct sum of compact operators tailing to zero, rewriting $A = \bigoplus_{i=1}^\infty A(s_i)$ on $\bigoplus_{i=1}^\infty H_{s_i}$. There the direct sum can clearly be approximated by operators of finite rank.

	For the `only if' direction, let $\Omega_1$ be the set of points of positive measure and let $\{M_k\}_{k=1}^\infty$ be the collection of sets in Proposition \ref{sim} of finite measure exhausting $\Omega$. Then
	$$S = \bigcup_{k=1}^\infty\bigcup_{n=1}^\infty \left\{s\in S\cap M_k : \mu\left(\{s\}\right)\geq\frac{\mu(M_k)}{n}\right\}.$$ Now $$\bigg|\left\{s\in S\cap M_k : \mu\left(\{s\}\right)\geq\frac{\mu(M_k)}{n}\right\} \bigg| \leq n, \quad (k,n\in\N),$$ so that $S$ is the countable union of finite sets. Writing $S = \{s_1,s_2,\dots\}$, it is clear that the restriction $\int_{S}^{\oplus}A(s)\,d\mu(s)$ of $A$ to the subspace $\int_{S}^\oplus H_s\,d\mu(s)$ is also compact and that we can write $$\int_{S}^\oplus H_s \,d\mu(s) = \bigoplus_{i=1}^\infty H_{s_i} \text{ and } \int_{S}^\oplus A(s)\,d\mu(s) = \bigoplus_{i=1}^\infty A(s_i).$$ It follows from the direct sum case that $A(s_i)$ is compact for all $i\in\N$ and $\|A(s_i)\|_{\B(H_{s_i})} \to 0$ as $i\to \infty$.
	
	If $A(s) =0$ for a.e.\ $s\in \Omega\setminus S$, we are done. If not, there exists $G \subset \Omega\setminus S$ such that $\mu(G) >0$ and $\|A(s)\|_{\B(H_s)} >0$ for all $s\in G$. Without loss of generality, we can assume by taking a smaller subset, that $\|A(s)\|_{\B(H_s)}\geq \varepsilon$ for some $\varepsilon>0$ and a.e.\ $s\in G$. Applying Lemma \ref{comlem} to the subspace $\int_G^\oplus H_s \,d\mu(s)$, we can add another point of positive measure $s_0 \in \Omega\setminus S$ to $S$. This contradicts that $S$ is the set of all points of positive measure and we are done.
\end{proof}

Thus, Theorem \ref{comthm} tells us that the only compact direct integrals of operators are, in fact, direct sums.

\section{Direct Integrals of $C_0$-Semigroups}\label{sec:4}

In this section, we begin by proving that we can take direct integrals of $C_0$-semigroups that have a uniform exponential growth bound to get a $C_0$-semigroup on the direct integral space. Furthermore, the generator of the direct integral $C_0$-semigroup is the direct integral of the generators.

\begin{theorem}\label{sgthm}
	Let $\big\{T^{(\cdot)}(s): \R_+ \to \B(H_s) : s \in \Omega\big\}$ be a collection of $C_0$-semigroups. If for each $t\geq0$, the direct integral of $\{T^t(s) : s\in \Omega\}$ exists as a bounded operator, then $T: \R_+ \to \B\left(\int_\Omega^\oplus H_s \, d\mu(s)\right)$ defined and denoted by
	$$T(t) := \int_\Omega^\oplus T^t(s)\,d\mu(s)$$ is a $C_0$-semigroup if and only if there exist $M \geq 1, \omega \in \R$ such that $$\|T^t(s)\|_{B(H_s)} \leq Me^{\omega t}$$ for a.e.\ $s \in \Omega$ and all $t >0$. In this case, $\|T(t)\| \leq Me^{\omega t}$ for all $t\geq 0$ and its generator $A$ is given by
	$$A = \int_\Omega^\oplus A(s)\,d\mu(s)$$ where $A(s)$ is the generator of $T^t(s)$ for a.e.\ $s \in \Omega$.
\end{theorem}

\begin{proof}
	First, we prove the `if' direction. $T : \R_+ \to B(\dih)$ clearly satisfies the semigroup property. We now check the strong continuity at $0$. Since $T^{(\cdot)}(s)$ is a $C_0$-semigroup for each $s \in \Omega$, for any $x \in \dih$, $T^t(s)x(s) \to x(s)$ in $H_s$ as $t\to 0$. Hence we have pointwise convergence of $\|T^t(s)x(s)-x(s)\|_s \to 0$. However,
	$$\|T^t(s)x(s)-x(s)\|_s^2 \leq \left(\|T^t(s)x(s)\|_s + \|x(s)\|_s\right)^2 \leq (Me^{\omega t}+1)^2\|x(s)\|_s^2$$ for all $t \geq0.$ Since $x \in \dih$, by Lebesgue's dominated convergence theorem,
	$$\int_\Omega \|T^t(s)x(s)-x(s)\|_s^2\, d\mu(s) \to 0.$$
	
	By Proposition \ref{bddprop}, $\|T(t)\| \leq Me^{\omega t}$ for all $t\geq0.$ Since any $C_0$-semigroup has an exponential bound, Proposition \ref{bddprop} also proves the `only if' direction.
	
	Finally, let $B$ be the generator of the direct integral $C_0$-semigroup $T$. Our goal is to prove that $B = A$, but we must first check that $A$ exists as a direct integral. Let $\lambda = \omega + 1$.  Then by the Laplace transform representation for resolvents of generators, $\lambda \in \rho(A(s))$ for a.e.\ $s$ and
	\begin{align*}
	\langle R(\lambda,A(s))x(s),y(s)\rangle_s & = \left\langle \int_0^\infty e^{-\lambda t} T^t(s)x(s)\,dt,y(s)\right\rangle_s \\ & = \int_0^\infty e^{-\lambda t} \langle T^t(s)x(s),y(s)\rangle_s \,dt
	\end{align*} for all $x,y \in \F$. Since the final integrand is measurable by assumption, Fubini's theorem tell us that $\{R(\lambda,A(s)) : s\in\Omega\}$ forms a measurable field of bounded operators. Hence $\{A(s) : s\in\Omega\}$ is a measurable field of closed operators.

	We now show that $B=A$. Again, let $\lambda = \omega +1$. Then $\lambda \in \rho(B)$, $\lambda \in \rho(A(s))$ and for some fixed $C>0$, $\|R(\lambda,B)\|, \|R(\lambda,A(s))\| \leq C$ for a.e.\ $s\in\Omega$. By Proposition \ref{specprop}, $\lambda \in \rho(A)$. Once again using the Laplace transform representation for resolvents of generators and Corollary \ref{specprop}(i), we have
	
	\begin{align}R(\lambda,B)x & = \int_0^\infty e^{-\lambda t}T(t)x \,dt \label{bochdih} = \int_0^\infty \bigg(\int_{\Omega}^\oplus e^{-\lambda t}T^t(s)\,d\mu(s)\bigg)x\, dt, \\
	R(\lambda,A)x & = \left(\int_\Omega^\oplus R(\lambda,A(s))\,d\mu(s)\right)x,\nonumber\end{align} and for a.e.\ $s\in\Omega$, \begin{equation}\label{bochHs} R(\lambda,A(s))x(s) = \int_0^\infty e^{-\lambda t}T^t(s)x(s) \,dt.\end{equation} The outer integral in Equation (\ref{bochdih}) is a Bochner integral in $\dih$ and the integral in Equation (\ref{bochHs}) is a Bochner integral in $H_s$. Let $x,y \in \dih$ be arbitrary. Then
	\begin{align}
	\left\langle \int_0^\infty \bigg(\int_\Omega^\oplus e^{-\lambda t}T^t(s)\,d\mu(s)\bigg)x\,dt,y\right\rangle \nonumber \\
	& \hspace{-5em} = \int_0^\infty \int_\Omega \left\langle e^{-\lambda t}T^t(s)x(s),y(s) \right\rangle_s \,d\mu(s)dt \nonumber \\ & \hspace{-5em} = \int_\Omega \int_0^\infty \left\langle e^{-\lambda t}T^t(s)x(s),y(s) \right\rangle_s \,dtd\mu(s) \label{fubinn} \\ & \hspace{-5em} = \int_\Omega \left\langle \int_0^\infty e^{-\lambda t}T^t(s)x(s)\,dt,y(s)\right\rangle_s\,d\mu(s) \nonumber \\ & \hspace{-5em} = \int_\Omega \big\langle R(\lambda,A(s))x(s),y(s)\big\rangle_s\,d\mu(s) \nonumber \\ & \hspace{-5em} = \left\langle \left(\int_{\Omega}^{\oplus}R(\lambda,A(s))\,d\mu(s)\right)x,y\right\rangle, \nonumber
	\end{align} where Equation (\ref{fubinn}) holds due to Fubini's theorem. Hence $\big\langle R(\lambda,B)x,y\big\rangle = \big\langle R(\lambda,A)x,y\big\rangle$ for arbitrary $x$ and $y$, so that $R(\lambda,B) = R(\lambda,A)$ and in particular, $B=A$.
\end{proof}

A natural question to ask is whether the sufficient conditions for a direct integral of operators to generate a $C_0$-semigroup are also necessary conditions (see \cite[Theorem 4.4]{LM}). A positive answer for this question is given in the following theorem.


\begin{theorem}\label{sgthmconv}
	If $A = \int_\Omega^\oplus A(s)\,d\mu(s)$ generates a $C_0$-semigroup $T(\cdot)$ on $\dih$, then $A(s)$ generates a $C_0$-semigroup denoted $T^{(\cdot)}(s)$ on $H_s$ for a.e.\ $s \in \Omega$ which has the same exponential bounds as the one generated by $A$. Furthermore,
	$$T(t) = \int_\Omega^\oplus T^t(s)\,d\mu(s).$$
\end{theorem}





\begin{proof} 
	If $A$ generates a $C_0$-semigroup on $\dih$, then there exist $M\geq 1,\omega \in \R$ such that $\|T(t)\| \leq Me^{\omega t}$ for all $t\geq0$ and for every $\lambda > \omega$, one has $\lambda \in \rho(A)$ and $$\|[(\lambda-\omega)R(\lambda,A)]^n\| \leq M$$ for all $n \in \N$. By Propositions \ref{specprop}(2) and \ref{bddprop}, for each $\lambda > \omega$ and $n \in \N$, there exists a measurable $U_{\lambda,n} \subset \Omega$ with $\mu(\Omega \setminus U_{\lambda,n})=0$ such that $\lambda \in \rho(A(s))$ and $$\|[(\lambda-\omega)R(\lambda,A(s))]^n\|_s \leq M$$ for all $s \in \Omega\setminus U_{\lambda,n}$. Let $G = \bigcup_{\lambda \in \Q_{\omega+}, n \in \N} U_{\lambda,n}$ where $\Q_{\omega+} = \{\mu \in \Q : \mu > \omega\}$. Then $\mu(\Omega\setminus G)=0$ and for all $s \in G$, we have $\Q_{\omega+} \subset \rho(A(s))$ and \begin{equation}\label{feller}\|[(\lambda-\omega)R(\lambda,A(s))]^n\|_s \leq M, \quad (\lambda \in \Q_{\omega+}, n \in \N).\end{equation}
	
	Clearly $\Q_{\omega+}$ is dense in $\{\mu \in \R : \mu > \omega\}$ so by considering bounded subsets of $\{\mu \in \R : \mu > \omega\}$ and applying Proposition \ref{denseresolvent} to the relation (\ref{feller}), we get that for all $\lambda \in \R$ with $\lambda > \omega$, we have $\lambda \in \rho(A(s))$ and by continuity, (\ref{feller}) holds for all $s\in G$. It follows that for all $s \in G$ and hence for a.e.\ $s \in \Omega$, $A(s)$ generates a $C_0$-semigroup $T^t(s)$ such that $$\|T^t(s)\|_{\B(H_s)} \leq Me^{\omega t}$$ for all $t\ge0$. 
	
	Let $\real \lambda > \omega$ so that $\lambda \in \rho(A(s))$ for a.e.\ $s \in \Omega$. Since $\{A(s) : s \in \Omega\}$ is a measurable field of closed operators, $\{R(\lambda,A(s)) : s \in \Omega\}$ forms a measurable field of bounded operators. Hence for any $x,y \in \F$, the Laplace transform representation for resolvents of generators gives us that
	\begin{align*}s\mapsto \big\langle R(\lambda,A(s))x(s),y(s)\big\rangle_s & = \left\langle \int_0^\infty e^{-\lambda t} T^t(s)x(s) \,dt,y(s)\right\rangle_s \\ &= \int_0^\infty e^{-\lambda t}\left\langle T^t(s)x(s),y(s) \right\rangle_s \,dt\end{align*} is measurable. Let
	$$F(\lambda,s) := \int_0^\infty e^{-\lambda t}\left\langle T^t(s)x(s) - x(s),y(s) \right\rangle_s \,dt,$$ which is measurable by Fubini's theorem. Notice that for fixed $s$, $F(\lambda,s)$ is the Laplace transform with respect to $t$ of $g(t,s)= \left\langle T^t(s)x(s) - x(s),y(s) \right\rangle_s$ which satisfies $g(0,s) = 0$, is continuous in $t$ by strong continuity of $T^t(s)$, and
	\begin{align*}|g(t,s)| & = \left|\left\langle T^t(s)x(s) - x(s),y(s) \right\rangle_s\right| \\ & \leq (ce^{\omega t}+1)\|x(s)\|_s \|y(s)\|_s \leq (c+1)\|x(s)\|_s \|y(s)\|_s e^{\omega' t}\end{align*}
	where $\omega' = \max{\{\omega,0\}}$ so that $g$ is exponentially bounded with respect to $t$. Applying Fubini's theorem two more times, we get that
	$$s \mapsto \frac{1}{u}\int_0^u \frac{1}{2\pi}\int_{\omega' +1-ir}^{\omega' + 1 + ir} e^{\lambda t}F(\lambda,s)\,d\lambda dr$$ is also measurable for each $t\geq0$. By \cite[Theorem 4.2.21(b)]{ABHN}, a Tauberian theorem,
	$$g(t,s)=\lim_{u\to \infty}\frac{1}{u}\int_0^u \frac{1}{2\pi}\int_{\omega' +1-ir}^{\omega' + 1 + ir} e^{\lambda t}F(\lambda,s)\,d\lambda dr$$ for a.e.\ $s\in \Omega$. Hence $s \mapsto \left\langle T^t(s)x(s) - x(s),y(s) \right\rangle_s$ is the pointwise a.e.\ limit of measurable functions for each $t\geq0$. This together with the uniform exponential bounds $\|T^t(s)\|_{\B(H_s)} \leq Me^{\omega t}$ implies that $\{T^t(s) : s \in \Omega\}$ form a measurable field of bounded operators for each $t \geq0$.
	
	Theorem \ref{sgthm} then implies that $A$ generates $\int_\Omega^\oplus T^t(s)\,d\mu(s)$, hence $$T(t) = \int_\Omega^\oplus T^t(s)\,d\mu(s).$$ \end{proof}

As mentioned in the introduction, Theorem \ref{sgthmconv} says that decomposable generators generate $C_0$-semigroups that are decomposable into individual $C_0$-semigroups.

\begin{remark}\label{decomp}
	If $A$ is a decomposable operator that generates a $C_0$-semigroup $T(\cdot)$, the following argument is an alternative way via the characterisation by Nussbaum to see that for each $t\ge0$, $T(t)$ is a decomposable operator. Since $A$ is decomposable, Corollary \ref{specprop}(i) implies that $R(\lambda,A)$ is decomposable for all $\lambda \in \rho(A)$. In particular, $R(\lambda,A)$ commutes with all bounded diagonalisable operators \cite[Corollary 4]{Nuss} and hence by \cite[Proposition 3.1.5]{ABHN}, $T(t)$ commutes with all bounded diagonalisable operators for any $t\ge0$. The characteristaion \cite[Corollary 4]{Nuss} again implies that $T(t)$ is decomposable for all $t\ge0$. However, unlike Theorem \ref{sgthmconv} this argument does not guarantee that $T(\cdot)$ decomposes into a family of $C_0$-semigroups, since all we would know from this is that $T(t)$ is decomposable for each $t\ge0$. Difficulties arise when trying to show the existence of a semigroup on $H_s$ for any particular $s$.
\end{remark}

We also have the following corollary which generalises \cite[Theorem 4.4]{LM} in the Hilbert space case in a weaker way than Theorem \ref{sgthmconv}.


\begin{corollary}\label{sgposmeascor}
	Suppose that $A = \int_\Omega^\oplus A(s)\,d\mu(s)$ generates a $C_0$-semigroup on $\dih$. Then for every measurable subset $\Omega'\subset \Omega$ with positive measure, $\int_{\Omega'}^\oplus A(s)\,d\mu(s)$ generates a $C_0$-semigroup on the closed subspace $\int_{\Omega'}^\oplus H_s \,d\mu(s)$.
\end{corollary}

\begin{proof}
	Theorem \ref{sgthmconv} gives the pointwise and uniform a.e.\ conditions for $\{A(s) : s\in\Omega'\}$ necessary to apply Theorem \ref{sgthm} with $\Omega'$ replacing $\Omega$.
\end{proof}

\begin{remark}
	A proof of this corollary can be obtained in an analogous way as the proof for \cite[Theorem 4.4]{LM}. The key step was to show that $\int_{\Omega'}^\oplus H_s \,d\mu(s)$ is invariant under $T$, the $C_0$-semigroup generated by $A$. Let $x \in \int_{\Omega'}^\oplus H_s\, d\mu(s)$. By the Post-Widder inversion formula \cite[Chapter III Corollary 5.5]{EN},
	$$T(t)x = \lim_{n\to \infty}\bigg[\frac{n}{t}R\Big(\frac{n}{t},A\Big)\bigg]^n x = \lim_{n\to \infty}\bigg[\int_\Omega^\oplus \frac{n}{t}R\Big(\frac{n}{t},A(s)\Big)\,d\mu(s)\bigg]^n x$$ for any $t\geq0$ where the resolvents in the integrand of the direct integral exist for a.e.\ $s \in \Omega$. By definition of the direct integral of operators, $$\bigg(\bigg[\int_\Omega^\oplus \frac{n}{t}R\Big(\frac{n}{t},A(s)\Big)\,d\mu(s)\bigg]^n x\bigg)(s)= 0$$ for a.e.\ $s \notin \Omega'$. Hence $T(t)x$ is the limit in $\dih$ of elements in $\int_{\Omega'}^\oplus H_s \,d\mu(s)$ and by closedness, is in $\int_{\Omega'}^\oplus H_s \,d\mu(s)$. We can apply \cite[Chapter II Proposition 2.3]{EN} to complete the proof.
	
	However, this method could not be directly used to prove Theorem \ref{sgthmconv} since $H_s$ does not embed into $\dih$ in any meaningful way and hence, like in Remark \ref{decomp}, difficulties potentially arise as to the existence of a semigroup for all $t\geq0$ on $H_s$ for any particular $s$.
\end{remark}

\section{Direct Integrals of Special Classes of Semigroups}\label{sec:5}

A natural question to ask is whether there are similar results to Theorems \ref{sgthm} and \ref{sgthmconv} for special classes of semigroups (see \cite[Chapter II Section 4]{EN}). Ahead of the following discussion, we note that these classes can be classified through spectral conditions in such a way that their direct integral theory once again, like in the general case, reduces to a.e.\ uniform conditions on the individual fibres.


\subsection{Bounded Analytic Semigroups}

First, we provide an easy consequence of Proposition \ref{specprop} that deals with sectoriality which we now define. A closed linear operator $A$ with dense domain $D(A)$ is \textit{sectorial of angle} $\delta \in [0,\frac{\pi}{2})$ if
\begin{enumerate}
	\item $\Sigma_{\pi/2 + \delta} \subset \rho(A)$, where $\Sigma_{\pi/2+\delta} = \left\{\lambda \in \C : |\arg \lambda | < \frac{\pi}{2}+\delta\right\}\setminus \{0\}$; and
	\item For each $\varepsilon \in (0,\delta)$, there exists $M_\varepsilon \geq 1$ such that
	$$\|R(\lambda,A)\|\leq\frac{M_\varepsilon}{|\lambda|}, \quad (0\neq \lambda \in \overline{\Sigma}_{\pi/2 + \delta - \varepsilon}).$$
\end{enumerate}

\begin{proposition}\label{secprop}
	Suppose that $A=\int_{\Omega}^\oplus A(s)\,d\mu(s)$. Then $A$ is sectorial of angle $\delta \in [0,\frac{\pi}{2})$ if and only if $\Sigma_{\pi/2+\delta}\setminus \{0\} \subset \rho(A(s))$ and for all $\varepsilon \in (0,\delta)$ there exists $M_\varepsilon \geq 1$ such that $$\|R(\lambda,A(s))\|_s \leq \frac{M_\varepsilon}{|\lambda|}, (\lambda \in \overline{\Sigma}_{\pi/2+\delta - \varepsilon}\setminus \{0\}),$$ for a.e.\ $s \in \Omega.$
\end{proposition}

\begin{proof}
	The `if' direction follows immediately from Corollary \ref{specprop}(i) followed by Proposition \ref{bddprop} as the uniform resolvent bound for a.e.\ $s$ given $\lambda \in \Sigma_{\pi/2 + \delta}$ guarantees that such $\lambda \in \rho(A)$.
	
	Now let $Q = \{a + ib \in \C : a,b \in\Q\}.$ The `only if' direction is then proved by showing, in the same way as the proof for Theorem \ref{sgthmconv}, that $\Sigma_{\pi/2+\delta}\cap Q\setminus \{0\} \subset \rho(A(s))$ and that for each $\varepsilon \in (0,\delta),$ the same constant $M_\varepsilon$ for bounding $\|\lambda R(\lambda,A)\|$ also bounds $\|\lambda R(\lambda,A(s))\|_s$ for all $\lambda \in \overline{\Sigma}_{\pi/2+\delta - \varepsilon} \cap Q \setminus \{0\}$ for a.e.\ $s \in \Omega.$ Again, as in the proof of Theorem \ref{sgthmconv}, we are done by Proposition \ref{denseresolvent} and continuity of the resolvent applied to compact subsets of the sectors.
\end{proof}

In other words, $A$ is sectorial of angle $\delta \in [0,\frac{\pi}{2})$ if and only if $A(s)$ is uniformly a.e.\ sectorial of angle $\delta$ for a.e.\ $s$. We restate this in terms of bounded analytic semigroups without proof (see \cite[Chapter II Section 4a.]{EN} and \cite[Chapter 3]{IK}).


\begin{theorem}\label{basgthm}
	Suppose that $A=\int_{\Omega}^\oplus A(s)\,d\mu(s)$. Then $A$ generates a bounded analytic semigroup $T$ of angle $\delta \in [0,\frac{\pi}{2})$ if and only if $A(s)$ generates an analytic semigroup $T^{(\cdot)}(s)$ of angle $\delta$ such that for every $0 < \delta' < \delta$, there exists $M_{\delta'} \geq 1$ with
	$$\|T^z(s)\|_s \leq M_{\delta'}, \quad (z \in \Sigma_{\delta'}),$$ uniformly for a.e.\ $s \in \Omega$.
\end{theorem}

These results can also be extended for general unbounded semigroups (see \cite[Propositions 3.16, 3.18, Theorem 3.19]{IK}). An analogous version of Proposition \ref{secprop} is also true for operators $A$ whose resolvent set contain sectors minus $\{0\}$ with angle smaller than $\pi/2$ and whose resolvent operators satisfy analogous bounds on these sectors. Our more restricted definition of sectoriality is due to the $C_0$-semigroup setting we are working in (see \cite[Chapter II, Definition 4.1]{EN}).

\subsection{Eventually Differentiable Semigroups}

In light of the characterisation of eventually differentiable semigroups found in \cite[Chapter II Theorem 4.14]{EN}, a similar result to Theorem \ref{basgthm} can be given for eventually differentiable semigroups.

\begin{theorem}\label{edsgthm}
	Suppose that $A=\int_{\Omega}^\oplus A(s)\,d\mu(s)$. Then $A$ generates an eventually differentiable semigroup exponentially bounded by $Me^{\omega t}$ for $\omega \in \R$ if and only if $A(s)$ generates a $C_0$-semigroup also exponentially bounded by $Me^{\omega t}$ and there exist constants $a,b,C>0$ with $$\Theta := \big\{\lambda \in \C : a e^{-b \real\lambda} \leq |\imag\lambda|\big\} \subset \rho(A(s))$$ and $$\|R(\lambda,A(s))\|_s \leq C|\imag\lambda|$$ for all $\lambda \in \Theta$ with $\real\lambda \leq \omega$ uniformly for a.e.\ $s\in \Omega$.
\end{theorem}

\begin{proof}
	The proof is almost exactly the same as that of Proposition \ref{secprop}. However, for the `if' direction, we must also check that $$\underset{s\in \Omega}{\esssup}\|R(\lambda,A(s))\|_{B(H_s)} < \infty$$ holds for all $\lambda \in \Theta$. For $\real\lambda \leq \omega$, the uniform a.e.\ resolvent bound is given by $C|\imag\lambda|.$ For $\real\lambda > \omega$,
	$$\|R(\lambda,A(s))\|_s \leq \frac{M}{\real\lambda - \omega}$$ for a.e.\ $s\in \Omega$ by virtue of standard semigroup generator properties.
\end{proof}

In other words, the direct integral operator generates an eventually differentiable semigroup if individual $A(s)$ are eventually differentiable in a uniform a.e.\ way.

\subsection{Immediately Norm-Continuous Semigroups}

We introduce the natural notion of norm-continuous in $t$ uniformly a.e.\ on $\Omega$ in order to formulate the most obvious result concerning immediately norm-continuous semigroups. We say that $\{T^t(s) : s \in \Omega\}$ is norm-continuous for $t > t_0$ uniformly a.e.\ on $\Omega$ if for every $t>t_0$ and $\varepsilon > 0$, there exists $\delta> 0$ such that
$$\|T^t(s)-T^{t_1}(s)\|_{B(H_s)} < \varepsilon, \quad (|t-t_1|<\delta),$$ for a.e.\ $s \in \Omega$. Note that $\delta$ is independent of $s$.

\begin{proposition}
	Suppose that $A=\int_{\Omega}^\oplus A(s)\,d\mu(s)$. Assume that $A(s)$ generates a semigroup $T^{(\cdot)}(s)$ for a.e.\ $s \in \Omega$ that are uniformly a.e.\ exponentially bounded and that $\{T^t(s) : s \in \Omega\}$ is norm-continuous for $t > t_0$ uniformly a.e.\ on $\Omega$. Then $A$ generates an eventually norm-continuous semigroup that is norm-continuous for $t>t_0$.
\end{proposition}

\begin{proof}
	The direct integral semigroup $T^{(\cdot)}$ exists by Theorem \ref{sgthm}. Let $t>t_0, \varepsilon > 0$, and $x \in \dih$. Then by norm-continuity for $t>t_0$ on $\Omega$ a.e.,\ there exists $\delta >0$ such that
	\begin{align*}
	\|(T^t-T^{t_1})x\|^2 & = \int_\Omega \|(T^t(s)-T^{t_1}(s))x(s)\|_s^2\, d\mu(s) \\ & \leq \int_\Omega \|T^t(s) - T^{t_1}(s)\|_s^2\|x(s)\|_s^2\, d\mu (s) \\ & < \varepsilon^2 \|x\|^2
	\end{align*} for all $|t-t_1| < \delta.$ Hence $\|T^t - T^{t_1}\| < \varepsilon$ for all $|t-t_1| < \delta$, proving that $T^t$ is norm-continuous for $t>t_0$.
\end{proof}

\cite[Chapter II Theorem 4.20]{EN} gives us a useful characterisation of immediately norm-continuous semigroups on Hilbert spaces, which we can use to prove the following result.

\begin{theorem}\label{incsgprop}
	Suppose that $A=\int_{\Omega}^\oplus A(s)\,d\mu(s)$. Then $A$ generates an immediately norm-continuous semigroup exponentially bounded by $Me^{-\varepsilon t}$ for some $\varepsilon>0$ if and only if $A(s)$ generates a $C_0$-semigroup $T^{(\cdot)}(s)$ also exponentially bounded by $Me^{-\varepsilon t}$ and
	$$\lim_{r \to \pm \infty}\|R(ir,A(s))\|_{\B(H_s)} =0$$ uniformly for a.e.\ $s\in \Omega$.
\end{theorem}

\begin{proof}
	For the `if' direction, Theorem \ref{sgthm} ensures that $A$ generates a $C_0$-semigroup that is exponentially bounded by $Me^{-\varepsilon t}$. By Proposition \ref{bddprop}, the fact that
	$$\lim_{r \to \pm \infty}\|R(ir,A)\| =0,$$ and \cite[Chapter II Theorem 4.20]{EN}, $A$ generates an immediately norm-continuous semigroup. For the converse, it is enough to see that as in the proof of Theorem \ref{sgthmconv}, there is a set $G\subset \Omega$ with $\mu(\Omega\setminus G) = 0$ such that  for all $s\in G,$ $i\Q \subset\rho(A(s))$ and $\|R(ir,A(s))\|_s \leq \|R(ir,A)\|$ for all $r \in \Q$. Continuity of the resolvent and density of $i\Q$ in $i\R$ completes the proof.
\end{proof}

\subsection{Immediately Compact Semigroups}

\cite[Chapter II Theorem 4.29]{EN} tells us that a $C_0$-semigroup $T$ is immediately compact if and only if $T$ is immediately norm-continuous and its generator $A$ has compact resolvent. For an operator $B$, the resolvent $R(\lambda,B)$ for any $\lambda \in \rho(B)$ cannot be zero. Hence, in light of Theorem \ref{comthm}, the only possible setting for a direct integral of semigroups to be immediately compact when $\mu$ is a Radon measure is the direct sum case. Combining \cite[Chapter II Theorem 4.29]{EN} with Theorem \ref{comthm} and Theorem \ref{incsgprop}, we get the following result with trivial proof.

\begin{theorem}
	Let $\mu$ be the counting measure on $\N$ and suppose that $A=\int_{\N}A(n)\,d\mu(n)$. Then $A$ generates an immediately compact semigroup if and only if $\displaystyle\lim_{r\to \pm \infty}\|R(ir,A(n))\|_{\B(H_n)} = 0$ uniformly for all $n\in\N$ and there exists $\lambda \in \C$ such that $\lambda \in \rho(A(n))$ and $R(\lambda,A(n))$ is compact for all $n\in\N$ and $\displaystyle\lim_{n\to \infty}\|R(\lambda,A(n))\|_{\B(H_n)} \to 0$. In particular, if $A$ generates an immediately compact semigroup, then so does $A(n)$ for all $n\in\N$.
\end{theorem}


\section{Examples and Applications}\label{sec:6}

In this section, we discuss two examples where Theorem \ref{sgthm} can be applied. Regrettably, these are somewhat special cases, where extra structure allows us to work more easily with the relevant spaces. The point, however, is that examples exist.

The first example is as follows. Let $\Omega=(0,1)$ and take $\mu$ to be the Lebesgue measure on $(0,1)$. It is worth informing the reader here that there is some ambiguity in the literature about what measures and measure spaces are used in direct integral theory. Most papers describe $\mu$ as a Borel measure on a locally compact topological space, as we have done so in Section \ref{sec:2}. This, however, does not mean we are only restricted to working with the Borel $\sigma$-algebra, as the known results hold on measure spaces that contain the Borel $\sigma$-algebra (see for example \cite{AC,GGST,Y}). Now define $\mathcal{H}$ by fixing $H_s$ to be the constant Hilbert space $L^2(\R)$, the space of square integrable complex-valued functions on the real line, for all $s\in(0,1)$. For the choice of measurable sections, we take $\F$ to be all sections $x:(0,1)\to \mathcal{H}$ such that
$$s\mapsto \langle x_s,g\rangle_{L^2}$$ is measurable for all $g\in L^2(\R)$. In this case, $\dih = L^2((0,1);L^2(\R)) = L^2((0,1)\times \R)$, the space of square integrable complex-valued functions on $(0,1)\times\R$. This is also known as the Hilbert space tensor product $L^2(0,1)\otimes L^2(\R)$.

Define the operator $A$ on $L^2(\R)$ by
\begin{align*} D(A) & = \{f\in L^2(\R) : f \text{ absolutely continuous}, f'\in L^2(\R)\} \\ Af & = f', \quad (f \in D(A)).
\end{align*}
This is widely known to generate the left-shift semigroup $T^{(\cdot)}$ given by $(T^tf)(r) = f(r+t)$ for $f\in L^2(\R)$ and $r\in\R$ (see for example \cite[Chapter II Section 2]{EN}). Now for each $s\in(0,1)$, define the operator $A_s$ by $D(A_s)=D(A)$ and $A_sf = (A-sI)f$ for $f\in D(A_s)$. Elementary differential equation theory shows that $\{A_s :D(A_s) \subset L^2(\R) \to L^2(\R) : s\in(0,1)\}$ is a measurable field of closed operators ($\sigma(A)=i\R$ so that $\sigma(A_s) = -s+i\R$ allowing for a common element of the resolvent sets to exist). Now each $A_s$ generates the re-scaled semigroup $T^{(\cdot)}_s := e^{- s\cdot}T^{(\cdot)}$ on $L^2(\R)$ and the family $\{T^{(\cdot)}_s : \R_+ \to \B(L^2(\R)): s\in (0,1)\}$ is uniformly bounded by $1$ as $T^{(\cdot)}$ is bounded by $1$. Futhermore for every $t>0$,
$$s\mapsto \langle e^{-st}T^t x_s , y_s \rangle_s = e^{-st}\int_{\R} x_s(r+t)\overline{y_s(r)} \, dr$$ is measurable for all $x,y \in \F$.

Thus by Theorem \ref{sgthm}, the direct integral $B := \int_{\Omega}^{\oplus} A_s \, d\mu(s)$ generates a $C_0$-semigroup $S^{(\cdot)}$ given by $S^t = \int_\Omega^\oplus T_s^t \, d\mu(s)$ for all $t\ge0$. Of course, one can also follow the same computation that is done in general terms in the proof of Theorem \ref{sgthm} to show this directly. In this simple example, $B$ can be thought of as encoding the range of a continuous pertubation of the differentiation operator. Note that it is essential that $\Omega$ be bounded below to ensure the re-scaled semigroups are uniformly exponentially bounded.

The second example is the case where the family of closed operators $\{A_s:D(A_s)\subset H_s \to H_s : s\in \Omega\}$ is a bounded measurable field of operators. In this case, exponentiation yields a family of uniformly bounded norm-continuous $C_0$-semigroups $\{\exp(\cdot A_s) : \R_+\to\B(H_s) : s\in\Omega\}$ (in fact, $C_0$-groups). From this point, we can follow two separate routes to obtain a decomposable $C_0$-semigroup on the direct integral space. The first is, of course, Theorem \ref{sgthm}. The second, more interesting method, is to apply functional calculus. We define the direct integral $A= \int_{\Omega}^{\oplus} A_s \, d\mu(s)$ which is also a bounded operator. From here, we can apply the Riesz-Dunford functional calculus to the exponential function $\exp(t\cdot)$ for each $t \ge0$. This is clearly holomorphic on an open neighbourhood containing $\sigma(A_s)$ for a.e.\ $s\in\Omega$. Using Gilfeather's result \cite[Theorem 1]{Gil}, we see that
$$\exp(tA) = \int_\Omega^\oplus \exp(tA_s) \, d\mu(s).$$

We end this section by mentioning the setting where $\mu$ is the counting measure. Here, the theory of direct integrals reduces to that of direct sums which is dealt with in \cite{LM}, where the theory is developed for Banach spaces more generally and not restricted to Hilbert spaces. In particular, an example of a stochastic particle system is used to show that the direct sum equivalent of Theorem \ref{sgthm}, \cite[Theorem 4.3]{LM}, automatically upgrades `separately weak' solutions to `strong solutions', resulting in an extra degree of differentiability for certian initial value problems.

It remains an interesting avenue of exploration to see if there are physical applications of Theorem \ref{sgthm} that require more than the scenario where the individual fibres of the Hilbert space bundle are equal to a fixed Hilbert space. It may be that more applications can be found, perhaps in the flavour of the stochastic example in \cite{LM} if our theory can be extended to the Banach space setting. In particular, the motivating link between $C_0$-semigroups and Cauchy problems lends itself to the hope that useful applications remain waiting to be found. Note that direct integrals of Banach spaces have been studied under the rubric of randomly normed spaces in \cite{HLR}.

\section{Asymptotics of Direct Integral Semigroups}\label{sec:7}

We motivate this final section by beginning with and stating a result concerning the stability of a uniformly bounded countable sequence of $C_0$-semigroups on Hilbert spaces from \cite{MaNa} with an added assumption that we believe is necessary. 

\begin{theorem}[{\cite[Theorem 3.2]{MaNa}}]\label{kekw}
	Let $\{T_n(t) : \R_+ \to \B(H_n) : n \in \N\}$ be a uniformly bounded sequence of $C_0$-semigroups with corresponding generators $A_n$ such that for all $n\in \N, i\R\subset \rho(A_n)$. Further assume that \begin{equation}\label{uni0}
	\sup_{n\in\N}\|A_n^{-1}\|_{H_n} < \infty.
	\end{equation}  Then for a fixed $\alpha >0$ the following conditions are equivalent:
	\begin{enumerate}[(i)]
		\item $\displaystyle\sup_{|r|\geq1,n\in\N} |r|^{-\alpha}\|R(ir,A_n)\|_{H_n} < \infty.$
		\item $\displaystyle\sup_{t\geq 0, n\in\N}\|t^{1/\alpha} T_n(t)A_n^{-1}\|_{H_n} < \infty.$
	\end{enumerate}
\end{theorem}

The assumption (\ref{uni0}) is missing from the statement of \cite[Theorem 3.2]{MaNa} as found in \cite{MaNa}. It is ambiguous as to whether or not the authors of that paper implicitly assume this condition in \cite[Theorem 3.2]{MaNa} without explicitly stating it. The introduction, preparatory work in Section 2, and application in Theorem 4.4 of their paper indicate that they are, in fact, assuming (\ref{uni0}). However, their statement of \cite[Proposition 3.6]{MaNa} seems to indicate otherwise. In any case, we both believe that the authors of \cite{MaNa} merely accidentally left out (\ref{uni0}) as an assumption in \cite[Theorem 3.2]{MaNa} and use the following example to show that the theorem fails if it is not assumed. We drop the subscripts for the norms as the context provides enough clarity.

\begin{example}
	Take $A_n = -1/n$ to be the multiplication operator acting on $H_n = \C$ for all $n\in \N$. Then $A_n$ generates the multiplication semigroup $T_n(t) = e^{-t/n}$ on $H_n$ for all $n\in \N$ which is uniformly bounded by $1$. Moreover,
	$ir-A_n = ir+1/n$ which has inverse $(ir+1/n)^{-1}$ for all $r\in \R, n\in \N$ so that $i\R\subset \rho(A_n)$ for all $n \in \N$. Condition (i) of Theorem \ref{kekw}  is satisfied since
	$$|r|^{-\alpha}\|R(ir,A_n)\| = |r|^{-\alpha}|ir + 1/n|^{-1} \leq |r|^{-(1+\alpha)}, \quad (n \in \N, |r| \geq 1).$$ However condition (ii) of Theorem \ref{kekw} is not satisfied since for $t=1$,
	$$\|1^{1/\alpha}T_n(1)A_n^{-1}\| = \|e^{-1/n}n\|\geq e^{-1}n$$ so that $\displaystyle\sup_{t\geq0, n\in \N}\|t^{1/\alpha}T_n(t)A_n^{-1}\| = \infty$.
\end{example}

Thus, (\ref{uni0}) cannot be omitted and furthermore, we believe it is just as natural to assume instead the stronger condition
\begin{equation}\label{ubia}\sup_{n\in\N}\|R(ir,A_n)\| < \infty, \quad (r\in\R),\end{equation} which we will explain after the following theorem which generalises \cite[Theorem 3.2]{MaNa} with the correct assumptions. We will prove the theorem using a much simpler argument than the one found in \cite{MaNa}, while still using \cite[Theorem 2.4]{BoTo} as in \cite{MaNa}.

\begin{theorem}\label{diap}
	Let $\{T^{(\cdot)}(s) : \R_+ \to \B(H_s) : s \in \Omega\}$ be a collection of $C_0$-semigroups with generators $A(s)$ such that for each $t\geq0,$ the direct integral of $\{T^t(s): s\in\Omega\}$ exists as a bounded operator on $H = \int_{\Omega}^\oplus H_s\, d\mu(s)$ and $$\underset{s\in \Omega}{\esssup} \sup_{t\geq0}\|T^t(s)\|_{H_s} <\infty.$$ Further assume that $i\R \subset\rho(A(s))$ for a.e.\ $s\in\Omega$ and $$\underset{s\in \Omega}{\esssup}\|R(ir,A(s))\|_{H_s} <\infty, \quad (r\in\R).$$ Then the following are equivalent:
	\begin{enumerate}[(i)]
		\item $\underset{s\in \Omega}{\esssup}\displaystyle\sup_{|r|\geq1}|r|^{-\alpha}\|R(ir,A(s))\|_{H_s} <\infty$.
		\item $\underset{s\in \Omega}{\esssup} \displaystyle\sup_{t\geq0}\|t^{1/\alpha}T^t(s)A(s)^{-1}\|_{H_s} <\infty.$
	\end{enumerate}
\end{theorem}

\begin{proof}
	Dropping the subscripts on the norms, there exists $M>0$ such that $\|T^t(s)\| \leq M$ for all $t\geq 0$ and a.e.\ $s$. Thus,
	$$T(t) = \int_\Omega^\oplus T^t(s)\, d\mu(s)$$ is a $C_0$-semigroup on $H$ bounded by $M$ with generator $$A = \int_{\Omega}^{\oplus} A(s)\, d\mu(s)$$ by Theorem \ref{sgthm}. Moreover, $i\R\subset\rho(A)$ by Corollary \ref{specprop}.
	
	Assuming (i), we see that \begin{equation}\label{diai}\displaystyle\sup_{|r|\geq 1}|r|^{-\alpha}\|R(ir,A)\| <\infty.\end{equation} Thus by \cite[Theorem 2.4]{BoTo},  \begin{equation}\label{diaii}\displaystyle\sup_{t\geq0}\|t^{1/\alpha}T(t)A^{-1}\|<\infty.\end{equation} Since
	$$T(t)A^{-1} = \int_{\Omega}^{\oplus}T^t(s)A(s)^{-1}\, d\mu(s),$$ (ii) follows. Now assuming (ii), we similarly have (\ref{diaii}), which implies (\ref{diai}) by \cite[Theorem 2.4]{BoTo} once again. Since
	$$R(ir,A) = \int_{\Omega}^{\oplus}R(ir,A(s))\, d\mu(s), \quad (r\in\R),$$ (i) follows.
\end{proof}

Thus, \cite[Theorem 3.2]{MaNa} with the additional assumption (\ref{ubia}) follows by taking the counting measure. In fact, the only theory necessary in this specific discrete case is that of direct sums, covered in \cite{LM}. We also see from the above proof that if we only assume $$\underset{s\in \Omega}{\esssup}\|A(s)^{-1}\|_{H_s} <\infty, \quad (s\in\Omega),$$ rather than the a.e.\ uniform resolvent bounds along the whole imaginary axis, we would still have $0\in\rho(A)$ and hence, in this case, if Theorem \ref{diap}(ii) holds, then (\ref{diaii}) holds as well. Hence \cite[Proposition 1.3]{BaDu} would imply that $i\R \subset \rho(A)$ and further that $$\underset{s\in \Omega}{\esssup}\|R(ir,A(s))\|_{H_s} <\infty, \quad (r\in\R),$$ demonstrating that the stronger condition (\ref{ubia}) is indeed just as natural to assume as (\ref{uni0}). 

The method of passing to a direct integral (or sum) in order to obtain the relation between uniform resolvent bounds and `uniform strong' stability can be applied to other quantified Tauberian theorems such as those found in \cite{BaDu,BoTo,RSS} (strongly continuous) and \cite{Sei16} (discrete).

\providecommand{\bysame}{\leavevmode\hbox to3em{\hrulefill}\thinspace}
\providecommand{\MR}{\relax\ifhmode\unskip\space\fi MR }
\providecommand{\MRhref}[2]{%
	\href{http://www.ams.org/mathscinet-getitem?mr=#1}{#2}
}
\providecommand{\href}[2]{#2}

\end{document}